\documentclass[11pt,twoside]{amsart}
\usepackage{amsmath, amsthm, amscd, amsfonts, amssymb, graphicx, color}
\usepackage[bookmarksnumbered, colorlinks, plainpages]{hyperref}

\newtheorem{thm}{Theorem}[section]
\newtheorem{cor}[thm]{Corollary}

\numberwithin{equation}{section}
\def\pn{\par\noindent}

\begin{document}

\leftline{ \scriptsize \it International Journal of Group Theory  Vol. {\bf\rm XX} No. X {\rm(}201X{\rm)}, pp XX-XX.}

\vspace{1.3 cm}

\title{On the Order of the Schur Multiplier of a Pair of Finite p-Groups II}
\author{ Fahimeh Mohammadzadeh,  Azam Hokmabadi$^*$ and Behrooz Mashayekhy}

\thanks{{\scriptsize
\hskip -0.4 true cm MSC(2010): Primary: 20E34; Secondary: 20D15.
$*$Corresponding author
\newline\indent{\scriptsize $\copyright$ 2011 University of Isfahan}}}

\maketitle

\begin{center}
Communicated by\;
\end{center}

\begin{abstract} Let $G$ be a finite $p$-group and $N$ be a normal subgroup of $G$, with
$|N|=p^n$ and $|G/N|=p^m$. A result of Ellis (1998) shows
that the order of the Schur multiplier of such a pair $(G,N)$ of finite $p$-groups is bounded
by $ p^{\frac{1}{2}n(2m+n-1)}$ and hence it is equal to $
p^{\frac{1}{2}n(2m+n-1)-t}$, for some non-negative integer $t$.
Recently the authors characterized the structure of $(G,N)$ when $N$ has a complement in $G$ and
$t\leq 3$. This paper is devoted to classify the structure of
$(G,N)$ when $N$ has a normal complement in $G$ and $t=4,5$.
\end{abstract}

\vskip 0.2 true cm


\pagestyle{myheadings}
\markboth{\rightline {\scriptsize  Mohammadzadeh,  Hokmabadi and Mashayekhy}}
         {\leftline{\scriptsize On the order of  Schur multipliers of pair of groups}}

\bigskip
\bigskip


\section{\bf Introduction}
\vskip 0.4 true cm

By a pair of groups $(G,N)$ we mean a group $G$ with a normal
subgroup $N$. In 1998 Ellis \cite{El} defined the Schur multiplier of a
pair $(G,N)$ to be the abelian group $M(G,N)$ appearing in a
natural exact sequence
\begin{eqnarray*}
 H_3(G) &\rightarrow& H_3(G/N) \rightarrow M(G,N) \rightarrow M(G)
\rightarrow M(G/N)\\
&\rightarrow& N/[N,G]\rightarrow (G)^{ab} \rightarrow (G/N)^{ab} \rightarrow 0,
\end{eqnarray*}
in which $H_3(G)$ is the third homology of $G$ with integer
coefficients. In 1956, Green \cite{G} showed that if $G$ is a group of
order $p^{n}$ , then its Schur multiplier is of order at most
$p^{\frac{n(n-1)}{2}}$, and hence equals to
$p^{\frac{n(n-1)}{2}-t}$, for some non-negative integer $t$. Berkovich \cite{B},  Zhou \cite{Z},  Ellis \cite{E} and  Niroomand \cite{N,Ni} determined
the structure of $G$ for $t=0,1,2,3,4,5$ by different methods.

 In 1998, Ellis \cite{El} gave
an upper bound for the order of the Schur multiplier of a pair of finite
$p$-groups. He proved that
 if $G$ is a finite $p$-group with normal subgroup $N$ of order $p^n$ and quotient $G/N$ of order $p^m$, then the Schur multiplier of  $(G,N)$ is bounded by $ p^{\frac{1}{2}n(2m+n-1)}$ and hence equals to $ p^{\frac{1}{2}n(2m+n-1)-t}$, for some nonnegative integer $t$.

 Let $(G,N)$ be a pair of
groups and $K$ be the complement of $N$ in $G$. In 2004, Sallemkar, Moghaddam and Saeedi \cite{SMS} characterized the structure of such a pair $(G,N)$ when $t=0,1$ with some conditions. Recently the authors \cite{HMM} determined the structure of the pair $(G,N)$, for $ t=0,1$, without any condition and also gave the structure of $(G,N)$ for $t=2,3$,
when $K$ is normal. In this paper we are going to determine the  structure of $(G,N)$  for $t=4,5$ when $K$ is a normal subgroup of $G$.

In this paper $D$ and  $Q$ denote the dihedral and quaternion
group of order 8 and, $E_1$ and $E_2$ denote the extra special
$p$-groups of order $p^3$ of odd exponent $p$
 and $p^2$, respectively. $E_4$ denotes the unique central product of a cyclic group of order $p^2$ and a non-abelian group of order $p^3.$
 Also ${\bf{Z}}_{n}^{(m)}$ denote the direct product of $m$
copies of ${\bf{Z}}_{n}$.\\

The following result is essential to prove the main theorems.\\

\begin{thm}\label{1}
\cite{El}Let $(G,N)$ be a pair of groups and
$K$ be the complement of $N$ in $G$. Then $$M(G) \cong M(G,N) \times M(K).$$
\end{thm}

In 1907 Schur \cite{K} gave an structure for the Schur multiplier of a
direct product of finite groups. He showed that $$M(G_ 1 \times G_2)=M(G_ 1) \times M(G_2)\times (G_1^{ab}\otimes G_2^{ab}). $$

As a  consequence  of this fact we have the following important result.\\

\begin{cor}\label{2}
 Let $(G,N)$ be a pair of groups and
$K$ be the complement of $N$ in $G$. Then
$$|M(G,N)|=|M(N)||N^{ab}\otimes K^{ab} |.$$
\end{cor}

The following theorems gave the structure of a finite $p$-group, in terms of the order of its Schur multiplier.\\

\begin{thm}\label{3} \cite{E}
 Let $G$ be a group of prime-power order $p^n$.
Suppose that $|M(G)|= p^{\frac{1}{2}n(n-1)-t}$. Then \\
$i)$ $t=0$ if and only if $G$ is elementary abelian;\\
$ii)$ $t=1$ if and only if $G \cong {\bf{Z}}_{p^2}$ or $G \cong E_1$;\\
$iii)$ $t=2$ if and only if $G \cong {\bf{Z}}_{p} \times {\bf{Z}}_{p^2}$, $G\cong D$
or $G\cong {\bf{Z}}_{p}\times E_1$;\\
$iv)$ $t=3$ if and only if $G\cong {\bf{Z}}_{p^3}$,
$G\cong {\bf{Z}}_{p}\times {\bf{Z}}_{p} \times {\bf{Z}}_{p^2}$, $G\cong Q$, $G\cong E_2$,
$G\cong D \times {\bf{Z}}_{2}$ or $G \cong E_1 \times {\bf{Z}}_{p} \times {\bf{Z}}_{p}$.
\end{thm}

\begin{thm}\label{6}  \cite{SMD} Let $G$ be an abelian group of order $p^n$.
Suppose that $|M(G)|= p^{\frac{1}{2}n(n-1)-4}$. Then $G$ is isomorphic to
${\bf{Z}}_{p^2}\times {\bf{Z}}_{p^2}$ or ${\bf{Z}}_{p^2}\times {{\bf{Z}}_{p}}^{(3)}$.\\
\end{thm}

The following result can be easily obtained by using a method similar to the proof of the above theorem in \cite{SMD}.

\begin{thm}\label{7} Let $G$ be an abelian group of order $p^n$.
Suppose that $|M(G)|= p^{\frac{1}{2}n(n-1)-5}$. Then $G$ is isomorphic to
${\bf{Z}}_{p^3}\times {\bf{Z}}_{p}$ or ${\bf{Z}}_{p^2}\times {{\bf{Z}}_{p}}^{(4)}$.\\
\end{thm}

\begin{thm}\label{4}  \cite{N} Let $G$ be a non-abelian group of order $p^n$.
Suppose that $|M(G)|= p^{\frac{1}{2}n(n-1)-4}$. Then $G$ is isomorphic to one of the following groups:\\
For $p=2$,\\
1) $D \times {\bf{Z}}_{p}^{(2)};$\\
2) $Q \times {\bf{Z}}_{2};$\\
3) $\langle a,b| a^4=b^4=1, [a,b,a]=[a,b,b]=1, [a,b]=a^2b^2 \rangle;$\\
4) $\langle a,b,c| a^2=b^2=c^2=1, abc=bca=cab \rangle;$\\
For $p\neq 2$\\
5) $E_4$;\\
6) $E_1 \times {\bf{Z}}_{p}^{(3)}$;\\
7) ${\bf{Z}}_{p}^{(4)} >\!\!\!\lhd _\theta {\bf{Z}}_{p}$;\\
8) $E_2 \times {\bf{Z}}_{p}$;\\
9) $\langle a,b| a^{p^2}=b^p=1, [a,b,a]=[a,b,b]=1 \rangle$;\\
10) $\langle a,b| a^9=b^3=1, [a,b,a]=1, [a,b,b]=a^6, [a,b,b,b]=1 \rangle;$\\
11) $\langle a,b| a^p=b^p=1, [a,b,a]=[a,b,b,a]=[a,b,b,b]=1 \rangle(p\neq 3);$\\
\end{thm}

\begin{thm}\label{5}
 \cite{Ni} Let $G$ be a non-abelian group of  order $p^n$.
Suppose that $|M(G)|= p^{\frac{1}{2}n(n-1)-5}$. Then $G$ is isomorphic to one of the following groups:\\
1) $D \times {\bf{Z}}_{2}^{(3)}$;\\
2) $E_1 \times {\bf{Z}}_{p}^{(4)}$;\\
3) $E_2 \times {\bf{Z}}_{p}^{(2)}$;\\
4) $E_4 \times {\bf{Z}}_{p};$\\
5) extra special $p$-group of order $p^5$;\\
6) $ \langle a,b| a^{p^2}=b^{p^2}=1, [a,b,a]=[a,b,b]=1, [a,b]=a^{p} \rangle$;\\
7) $\langle a,b| a^{p^2}=b^p=1, [a,b,a]=[a,b,b]=a^{p}, [a,b,b,b,]=1 \rangle$;\\
8) $\langle a,b| a^{p^2}=b^p=1, [a,b,a]=[a,b,b,b,]=1, [a,b,b]=a^{np} \rangle$ where $n$ is a fixed quadratic non-residue of $p$ and $p\neq 3$;\\
9) $\langle a,b| a^{p^2}=1, b^3=a^3,, [a,b,a]=[a,b,b,b,]=1, [a,b,b]=a^{6} \rangle$;\\
10) $\langle a,b| a^{p}=1, b^p=[a,b,b], [a,b,a]=[a,b,b,b,]=[a,b,b,a]=1 \rangle$;\\
11) $D_{16}$;\\
12) $\langle a,b| a^{4}= b^4=1, a^{-1}ba=b^{-1} \rangle$;\\
13)  $ Q \times {\bf{Z}}_{2}^{(2)}$;\\
 14) $(D \times {\bf{Z}}_{2}) >\!\!\!\lhd {\bf{Z}}_{2}$;\\
15) $(Q \times {\bf{Z}}_{2}) >\!\!\!\lhd {\bf{Z}}_{2}$;\\
16) ${\bf{Z}}_{2} \times \langle a,b,c| a^{2}= b^2=c^2=1, abc=bca=cab \rangle ;$\\
\end{thm}


\section{\bf {\bf \em{\bf Main Results}}}
\vskip 0.4 true cm

In this section, let $(G,N)$ be a pair of
groups such that $G\cong N\times K$, with $|N|=p^n$ and $|K|= p^m$. It is obtained that $|M(G,N)|= p^{\frac{1}{2}n(2m+n-1)-t}$, for some nonnegative integer $t$. Recently, all these pairs of finite $p$-groups are listed in \cite{HMM} by the authors, when $t=0,1,2,3$.
The aim of this paper is characterizing the structure of such pairs of finite $p$-groups, when $t=4,5$. \\

\begin{thm}
 By the above assumption, $t=4$ if and only if  $G$ is isomorphic to one of the following groups:\\
1) $G\cong N\times K$ where $N\cong {\bf{Z}}_{p}$ and $K$ is any
group with $d(K)=m-4$;\\
2) $G\cong N\times K$ where $N\cong {\bf{Z}}_{p} \times
{\bf{Z}}_{p}$ and $K$ is any group with $d(K)=m-2$;\\
3) $G\cong N\times K$ where $N\cong {{\bf{Z}}_{p}}^{(4)}$ and $K$
is any group with $d(K)=m-1$;\\
4) $G= N\cong D \times {\bf{Z}}_{2}^{(2)};$\\
5) $G= N\cong Q \times {\bf{Z}}_{2};$\\
6) $G= N\cong  \langle a,b| a^4=b^4=1, [a,b,a]=[a,b,b]=1, [a,b]=a^2b^2 \rangle;$\\
7) $G= N\cong \langle a,b,c| a^2=b^2=c^2=1, abc=bca=cab \rangle;$\\
8) $G= N\cong E_4$;\\
9) $G= N\cong E_1 \times {\bf{Z}}_{p}^{(3)}$;\\
10) $G= N\cong {\bf{Z}}_{p}^{(4)}  >\!\!\!\lhd _\theta {\bf{Z}}_{p}$;\\
11) $G= N\cong  E_2 \times {\bf{Z}}_{p}$;\\
12) $G= N\cong \langle a,b| a^{p^2}=b^p=1, [a,b,a]=[a,b,b]=1 \rangle$;\\
13) $G=N\cong \langle a,b| a^9=b^3=1, [a,b,a]=1, [a,b,b]=a^6, [a,b,b,b]=1 \rangle;$\\
14) $G= N\cong \langle a,b| a^p=b^p=1, [a,b,a]=[a,b,b,a]=1, [a,b,b,b]=1 \rangle(p\neq 3);$\\
15) $G= N\cong {\bf{Z}}_{p^2}\times {\bf{Z}}_{p^2};$\\
16) $G= N\cong {\bf{Z}}_{p^2}\times
{{\bf{Z}}_{p}}^{(3)};$\\
17) $G\cong N\times K$ where  $K={\bf{Z}}_{p}$ and $N\cong
{{\bf{Z}}_{p}}^{(2)}\times {\bf{Z}}_{p^2};$ \\
18) $G\cong N\times K$ where  $K={\bf{Z}}_{p}$ and $N\cong Q;$\\
19) $G\cong N\times K$ where  $K={\bf{Z}}_{p}$ and $N\cong E_2;$\\
20) $G\cong N\times K$ where  $K={\bf{Z}}_{p}$ and $N\cong D\times
{\bf{Z}}_{2}; $\\
21) $G\cong N\times K$ where  $K={\bf{Z}}_{p}$ and $N\cong
E_1\times
{{\bf{Z}}_{p}}^{(2)} $;\\
22) $G\cong N\times K$ where  $K={{\bf{Z}}_{p}}^{(2)}$ and $N\cong
{\bf{Z}}_{p^2}\times
{\bf{Z}}_{p} $;\\
23) $G\cong N\times K$ where  $K={{\bf{Z}}_{p}}^{(2)}$ and $N\cong
D $;\\
24) $G\cong N\times K$ where  $K={{\bf{Z}}_{p}}^{(2)}$ and $N\cong
E_1\times
{\bf{Z}}_{p} $;\\
25) $G\cong N\times K$ where  $K={\bf{Z}}_{p^2}\times {\bf{Z}}_{p}$
and $N\cong {\bf{Z}}_{p^2} $;\\
26) $G\cong N\times K$ where  $K={{\bf{Z}}_{p}}^{(3)}$
and $N\cong E_1 $;\\
27) $G\cong N\times K$ where  $K={{\bf{Z}}_{p}}^{(3)}$
and $N\cong {\bf{Z}}_{p^2}$.\\
\end{thm}

\begin{proof}  The necessity of theorem follows from the fact that $G= N\times K$ and Corollary \ref{2}.
 For sufficiency first suppose that $N$ is an elementary abelian $p$-group. Then $nm-4=nd(K)$.
 Since $|N \otimes K^{ab}|=p^{nm-4}$ by Corollary  \ref{2} and also it is of order $p^{nd(K)}$.
So $n(m-d(K))=4$, which implies that $n=1,2$ or 4. Therefore  $N\cong {\bf{Z}}_{p}$
  and $K$ is any group with $d(K)=m-4$ or $N\cong {{\bf{Z}}_{p} }^{(2)}$ and
  $K$ is any group with $d(K)=m-2$, or $N\cong {{\bf{Z}}_{p} }^{(4)}$ and
  $K$ is any group with $d(K)=m-1$.

Now suppose that $N$ is not an elementary abelian $p$-group .
Then using Corollary  \ref{2}, we have
$|N^{ab}\otimes K^{ab}|>p^{nm-4}$ and so $md(N)>nm-4$  which
implies that  $m(n-d(N))<4.$ Therefore $m=0,1,2,3$ .

If $m=0$,
then $K=1$ and $N$ is one of the groups which are listed in
Theorems \ref{6} and \ref{4}.

 If $m=1$, then $K={\bf{Z}}_{p}$ and $d(N)=n-1, n-2$
or $n-3$, respectively. It follows that $|N^{ab}\otimes
K|=p^{n-1}, p^{n-2}$ or $p^{n-3}$. So Corollary  \ref{2} implies that
$|M(N)|= p^{\frac{n^2-n}{2}-3}, p^{\frac{n^2-n}{2}-2}$, or
$p^{\frac{n^2-n}{2}-1}$. In the first case $N$ is $
{\bf{Z}}_{p}\times {\bf{Z}}_{p} \times {\bf{Z}}_{p^2}, Q, E_2, ,
D \times {\bf{Z}}_{2}$ or $ E_1 \times {\bf{Z}}_{p} \times
{\bf{Z}}_{p}$  and other cases are impossible, by Theorem \ref{3}.

 If $m=2$, then  $d(N)=n-1$ and
$K={\bf{Z}}_{p}\times{\bf{Z}}_{p}$ or $K={\bf{Z}}_{p^2}$. In the
first case $|N^{ab}\otimes K|=p^{2(n-1)}$ and so $|M(N)|=
p^{\frac{n^2-n}{2}-2}$. Therefore $N$ is ${\bf{Z}}_{p} \times
{\bf{Z}}_{p^2}, D$ or $ {\bf{Z}}_{p}\times E_1$. In the second
case $N^{ab}\cong {{\bf{Z}}_{p}}^{(n-1)}$ or $N^{ab}\cong
{\bf{Z}}_{p^2}\times {{\bf{Z}}_{p}}^{(n-2)}$. If $N^{ab}$ is an
elementary abelian $p$-group, then $|N^{ab}\otimes K|=p^{(n-1)}$
and so $|M(N)|= p^{\frac{n^2+n-6}{2}}$ which is impossible and
if  $N^{ab}\cong {\bf{Z}}_{p^2}\times {{\bf{Z}}_{p}}^{(n-2)}$,
then $|M(N)|= p^{\frac{n^2+n-8}{2}}$ which is impossible too.

 If $m=3$,
then  $d(N)=n-1$ and $K$ is an abelian $p$-group of order $p^3$
or an extra special $p$-group of order $p^3$. In the first
case we have three possibility for $K$. The first possibility is
$K\cong {\bf{Z}}_{p}^{{(3)}},$. Then similar to the previous part, one
can see that  $|M(N)|= p^{\frac{n^2-n}{2}-1}$ and so $N\cong E_1$
or ${\bf{Z}}_{p^{2}}$. The second possibility is $K\cong
{\bf{Z}}_{p^{3}}$. This implies that $n=1$ which is a
contradiction. The third possibility is $K \cong
{\bf{Z}}_{p^2}\times {\bf{Z}}_{p}.$ which implies that $n=2$ and $N\cong
{\bf{Z}}_{p^2}$.\\
In the second case,  if $K$ is an extra special $p$-group of order
$p^3$, then $N^{ab}\cong {\bf{Z}}_{p}^{(n-1)}$ or $N^{ab}\cong
{\bf{Z}}_{p^2}\times {\bf{Z}}_{p}^{(n-2)}$. This implies that
$|N^{ab}\otimes K|=|N^{ab}\otimes
{\bf{Z}}_{p}^{(2)}|=p^{2n-2}$ and so $n=1$ which is a contradiction.
 Hence the proof is complete.
\end{proof}

\begin{thm} By the previous assumption,  $t=5$ if and only if  $G$ is isomorphic to one of the following groups:\\
1) $G\cong N\times K$ where $N\cong {\bf{Z}}_{p}$ and $K$ is any
group with $d(K)=m-5$;\\
2) $G\cong N\times K$ where $N\cong {{\bf{Z}}_{p}}^{(5)}$ and $K$
is any
group with $d(K)=m-1$;\\
3) $G= N\cong D \times {\bf{Z}}_{2}^{(3)}$;\\
4) $G= N\cong E_1 \times {\bf{Z}}_{p}^{(4)}$;\\
5) $G= N\cong E_2 \times {\bf{Z}}_{p}^{(2)}$;\\
6) $G= N\cong  E_4 \times {\bf{Z}}_{p};$\\
7) $G= N\cong$ extra special $p$-group of order $p^5$;\\
8) $G= N\cong \langle a,b| a^{p^2}=b^{p^2}=1, [a,b,a]=[a,b,b]=1, [a,b]=a^{p} \rangle$;\\
9) $G= N\cong \langle a,b| a^{p^2}=b^p=1, [a,b,a]=[a,b,b]=a^{p}, [a,b,b,b,]=1 \rangle$;\\
10) $G= N\cong \langle a,b| a^{p^2}=b^p=1, [a,b,a]=[a,b,b,b,]=1, [a,b,b]=a^{np}\rangle$ where $n$ is a fixed quadratic non-residue of $p$ and $p\neq 3$ ;\\
11) $G= N\cong \langle a,b| a^{p^2}=1, b^3=a^3,[a,b,a]=[a,b,b,b,]=1, [a,b,b]=a^{6}\rangle$;\\
12) $G= N\cong \langle a,b|a^{p}=1,b^p=[a,b,b],[a,b,a]=[a,b,b,b,]=[a,b,b,a]=1 \rangle$;\\
13) $G= N\cong D_{16}$;\\
14) $G= N\cong \langle a,b| a^{4}= b^4=1, a^{-1}ba=b^{-1} \rangle $;\\
15) $G= N\cong Q \times {\bf{Z}}_{2}^{(2)}$;\\
16) $G= N\cong (D \times {\bf{Z}}_{2}) >\!\!\!\lhd{\bf{Z}}_{2}$;\\
17) $G= N\cong (Q \times {\bf{Z}}_{2}) >\!\!\!\lhd{\bf{Z}}_{2}$;\\
18) $G= N\cong {\bf{Z}}_{2} \times <a,b,c| a^{2}= b^2=c^2=1, abc=bca=cab>;$\\
19) $G= N\cong {\bf{Z}}_{p^3}\times {\bf{Z}}_{p}$;\\
20) $G= N\cong {\bf{Z}}_{p^2}\times {{\bf{Z}}_{p}}^{(4)};$\\
21) $G\cong N\times K$ where  $K={{\bf{Z}}_{p}}$ and $N\cong
D\times { {\bf{Z}}_{2}}^{(2)} $;\\
22) $G\cong N\times K$ where  $K={{\bf{Z}}_{p}}$ and $N\cong
Q\times { {\bf{Z}}_{2}} $;\\
23) $G\cong N\times K$ where  $K={{\bf{Z}}_{p}}$ and $N\cong
<a,b,c| a^{2}= b^2=c^2=1, abc=bca=cab> $;\\
24) $G\cong N\times K$ where  $K={{\bf{Z}}_{p}}$ and $N\cong
E_4 $;\\
25) $G\cong N\times K$ where  $K={{\bf{Z}}_{p}}$ and $N\cong
E_1\times { {\bf{Z}}_{p}}^{(3)} $;\\
26) $G\cong N\times K$ where  $K={{\bf{Z}}_{p}}$ and $N\cong
{\bf{Z}}_{p}^{(4)}  >\!\!\!\lhd  {\bf{Z}}_{p}$;\\
27) $G\cong N\times K$ where  $K={{\bf{Z}}_{p}}$ and $N\cong  E_2 \times {\bf{Z}}_{p}$;\\
28) $G\cong N\times K$ where  $K={{\bf{Z}}_{p}}$ and $N\cong
E_2\times { {\bf{Z}}_{p}} $;\\
29) $G\cong N\times K$ where  $K={{\bf{Z}}_{p}}^{(2)}$ and $N\cong
E_1\times { {\bf{Z}}_{p}}^{(2)} $;\\
30) $G\cong N\times K$ where  $K={{\bf{Z}}_{p}}^{(2)}$ and $N\cong
{ {\bf{Z}}_{p}}^{(2)}\times { {\bf{Z}}_{p^2}} $;\\
31) $G\cong N\times K$ where  $K={{\bf{Z}}_{p}}^{(2)}$ and $N\cong
Q $;\\
32) $G\cong N\times K$ where  $K={{\bf{Z}}_{p}}^{(2)}$ and $N\cong
E_2 $;\\
33) $G\cong N\times K$ where  $K={{\bf{Z}}_{p}}^{(2)}$ and $N\cong
D\times { {\bf{Z}}_{2}}; $\\
34) $G\cong N\times K$ where  $K={{\bf{Z}}_{p^2}}$ and $N\cong
E_1$;\\
35) $G\cong N \times K$ where  $K={{\bf{Z}}_{p^2}}$ and $N\cong
{ {\bf{Z}}_{p}}\times { {\bf{Z}}_{p^2}} $;\\
36) $G\cong N\times K$ where  $K={{\bf{Z}}_{p^3}}$ and $N\cong
 { {\bf{Z}}_{p^2}} $;\\
37) $G\cong N\times K$ where  $K={{\bf{Z}}_{p}}^{(3)}$ and $N\cong
D $;\\
38) $G\cong N \times K$ where  $K={{\bf{Z}}_{p}}^{(3)}$ and $N\cong
E_1 \times {{\bf{Z}}_{p}} $;\\
39) $G\cong N\times K$ where  $K={{\bf{Z}}_{p}}^{(3)}$ and $N\cong
{ {\bf{Z}}_{p}}\times { {\bf{Z}}_{p^2}} $;\\
40) $G\cong N\times K$ where  $K$ is an extra special $P$-group
and $N\cong
{ {\bf{Z}}_{p^2}} $;\\
41) $G\cong N\times K$ where  $K={{\bf{Z}}_{p}}^{(4)}$ and $N\cong
E_1 $,\\
42) $G\cong N\times K$ where  $K={{\bf{Z}}_{p}}^{(2)}\times {
{\bf{Z}}_{p}}$ and $N\cong
{ {\bf{Z}}_{p^2}}. $\\
\end{thm}

\begin{proof} The proof of this theorem is similar to the proof
of the previous theorem so we left the details to the reader.
Necessity is straight forward. For sufficiency first suppose that $N$ is an elementary abelian $p$-group . Then
 $n(m-d(K))=5$, so $n=1$ or 5. If $n=1$, then $N\cong
 {\bf{Z}}_{p}$ and $K$ is any group with $d(K)=m-5$ and $n=5$
 implies that $N\cong
 {\bf{Z}}_{p}^{(5)}$ and $K$ is any group with $d(K)=m-1.$

Suppose that $N$ is not an elementary abelian $p$-group. Then we have
$|N^{ab}\otimes K^{ab}|>p^{nm-5}$, by Corollary  \ref{2}. It follows
that $md(N)>nm-5$ and thus $m(n-d(N))<5$, which implies that
$m=0,1,2,3$ or $4.$  If $m=0$, then $K=1$ and $N$ is one of the
groups that are listed in Theorems \ref{7} and \ref{5}.

 If $m=1$, then $K={\bf{Z}}_{p}$
and $d(N)=n-i$, for $1 \leq i \leq 4$. Therefore by Corollary
 \ref{2}, $|M(N)|= p^{\frac{n(n-1)}{2}-(5-i)}$ for $1 \leq i \leq 4$,
respectively. It follows that $N\cong {\bf{Z}}_{p^2}\times
{{\bf{Z}}_{p}}^{(3)}$, $D \times {\bf{Z}}_{2}^{(2)}$, $Q \times
{\bf{Z}}_{2}$, $E_4$, $E_1 \times
{\bf{Z}}_{p}^{(3)}$, $E_2 \times {\bf{Z}}_{p}$, $N\cong{\bf{Z}}_{p^3}$ or $\langle a,b,c| a^2=b^2=c^2=1, abc=bca=cab \rangle$, by Theorems  \ref{4}, \ref{6} and \ref{3}.

 If $m=2$, then
$K={\bf{Z}}_{p^2}$ or $K={\bf{Z}}_{p}\times{\bf{Z}}_{p}$ and
$d(N)=n-1$
 or $d(N)=n-2$. First suppose that $K={\bf{Z}}_{p}\times{\bf{Z}}_{p}$ . If $d(N)=n-1$, then $|M(N)|= p^{\frac{n^2-n}{2}-3}$.
 It follows that $N\cong  {\bf{Z}}_{p}\times {\bf{Z}}_{p} \times
 {\bf{Z}}_{p^2}$, $N \cong E_1 \times {\bf{Z}}_{p} \times {\bf{Z}}_{p}$,
$N\cong Q$, $N\cong E_2$ or $N\cong D \times {\bf{Z}}_{2}$. If
$d(N)=n-2$, then $|N\otimes K|=p^{2(n-2)}$. This implies that
$|M(N)|= p^{\frac{n^2+n-2}{2}}$ and so $n<1$ which is
impossible.\\
Now suppose that $K={\bf{Z}}_{p^2}$. If $d(N)=n-1$ and
$N^{ab}\cong {{\bf{Z}}_{p}}^{(n-1)}$, then $|M(N)|=
p^{\frac{n^2+n-8}{2}}$ which implies that $n=3$ and $|M(N)|=
p^{2}$. Therefore $N\cong E_1$. If $d(N)=n-1$ and $N^{ab}\cong
{\bf{Z}}_{p}^2\times  {{\bf{Z}}_{p}}^{(n-2)}$, then $ n=3$ or $n=4$. For $n=3$ there is not any structure for $N$ and $n=4$ implies that  $N\cong {\bf{Z}}_{p^2}\times  {\bf{Z}}_{p}$.\\
If $d(N)=n-2$, then $N^{ab}\cong {\bf{Z}}_{p}^{(n-2)}$ or
$N^{ab}\cong {\bf{Z}}_{p}^3\times  {\bf{Z}}_{p}^{(n-3)}$ or
$N^{ab}\cong {\bf{Z}}_{p}^2\times {\bf{Z}}_{p}^2\times
{\bf{Z}}_{p}^{(n-4)}$. Therefore similar to the previous case one can see that $n<5$ which is
impossible.

If $m=3$, then  $d(N)=n-1$ and $K$ is an abelian $p$-group of
order $p^3$ or is an extra special $p$-group of order $p^3$. In
the first case we have three possibility for $K$. The first
possibility is $K\cong {\bf{Z}}_{p^3}$. Now if $N^{ab}\cong
{\bf{Z}}_{p}^{(n-1)}$, then $n=1$ which is impossible and if $N^{ab}\cong {\bf{Z}}_{p^2}\times
{\bf{Z}}_{p}^{(n-2)}$, then $N\cong {\bf{Z}}_{p^2}$. The second possibility is $K\cong
{\bf{Z}}_{p^2}\times {\bf{Z}}_{p}.$ In this case there is no
structure for $N$. The third possibility is $K\cong
{\bf{Z}}_{p}^{(3)}$. Thus $|M(N)|= p^{\frac{n^2-n}{2}-2}$ and so
$N\cong D$, $N\cong E_1 \times {\bf{Z}}_{p}$  or $N\cong
{\bf{Z}}_{p^2}\times
{\bf{Z}}_{p}$, by Theorem  \ref{3}.\\
 Now suppose that $K$ is an extra special $p$-group of order
 $p^3$. Then $N^{ab}\cong
{\bf{Z}}_{p}^{(n-1)}$ or $N^{ab}\cong {\bf{Z}}_{p^2}\times
{\bf{Z}}_{p}^{(n-2)}$. If $N^{ab}$ is an elementary abelian, then
there is no structure for $N$.  Otherwise $N={\bf{Z}}_{p^2}$.

If $m=4$, then  $d(N)=n-1$. So $N^{ab}\cong {\bf{Z}}_{p}^{(n-1)}$
or $N^{ab}\cong {\bf{Z}}_{p}^2\times {\bf{Z}}_{p}^{(n-2)}$. In
the first case $|N^{ab}\otimes
K^{ab}|=|{\bf{Z}}_{p}^{(n-1)}\otimes K^{ab}|\leq p^{d(K)(n-1)}$.
Now suppose that $d(K)< 4$. Then we have $|N^{ab}\otimes K^{ab}|\leq
p^{3(n-1)}$. Therefore  $|M(N)| \geq p^{-3(n-1)+(8n+n^2-n-10)/2}$,
by Corollary  \ref{2}. So $n<2$ which is impossible. If $d(K)=4$, then we have
$K\cong {\bf{Z}}_{p}^{(4)}$. Hence $|N^{ab}\otimes
K^{ab}|=p^{4(n-1)}$ which implies that $|M(N)| =p^{(n^2-n)/2-1}$.
So $N\cong E_1$, by Theorem  \ref{3}. \\
In the second case, suppose that $K$ is not abelian. Then  $|N^{ab}\otimes
K^{ab}|=|{\bf{Z}}_{p^2}\times {\bf{Z}}_{p}^{(n-2)}\otimes
K^{ab}|=|{\bf{Z}}_{p^2}\otimes K^{ab}||{\bf{Z}}_{p}^{(n-2)}\otimes
K^{ab}|\leq p^3p^{d(K)(n-2)}\leq p^{3(n-1)}$ which implies that
$n<2$ and it is impossible.\\
If $K$ is abelian, then $K\cong {\bf{Z}}_{p^2}\times
{\bf{Z}}_{p}^{(2)}$. thus $|N^{ab}\otimes K|={p}^{3n-2}$. It
follows that $N\cong {\bf{Z}}_{p^2}$. This completes the proof.
\end{proof}

\vskip 0.4 true cm



\bigskip
\bigskip


{\footnotesize \pn{\bf Fahimeh Mohammadzadeh}\; \\ {Department of Mathematics},
{Payame Noor University} {19395-4697 Tehran, Iran}\\
{\tt Email: F.mohamadzade@gmail.com}\\

{\footnotesize \pn{\bf Azam Hokmabadi}\; \\ {Department of Mathematics},
{Payame Noor University} {19395-4697 Tehran, Iran}\\
{\tt Email: ahokmabadi@pnu.ac.ir}\\

{\footnotesize \pn{\bf Behrooz Mashayekhy}\; \\ {Department of
Mathematics}, {Ferdowsi University of Mashhad, P.O.Box 1159-91775,} {Mashhad, Iran}\\
{\tt Email: bmashf@um.ac.ir}\\
\end{document}